\documentclass[12pt,a4paper,reqno]{amsart} 
\usepackage[T1]{fontenc}       
\usepackage{amsfonts}          
\usepackage{amsmath}
\usepackage{amssymb}
\usepackage{overpic}
\usepackage{euscript} 
\usepackage{eurosym}
\usepackage{comment}
\usepackage{mathrsfs} 
\usepackage{ifthen}
\usepackage{esint} 
\usepackage{color}

\long\def\symbolfootnote[#1]#2{\begingroup%
\def\thefootnote{\fnsymbol{footnote}}\footnote[#1]{#2}\endgroup}

{\qed\vspace{5pt}}

\usepackage{amsthm}

\newtheoremstyle{lause}
{5pt}
{5pt}
{\slshape}
{\parindent}
{\bfseries}
{.}
{.5em}
{}

\theoremstyle{lause}

\newtheoremstyle{maaritelma}
{5pt}
{5pt}
{\rmfamily}
{\parindent}
{\bfseries}
{.}
{.5em}
{}

\theoremstyle{maaritelma}
\newtheoremstyle{lause}
{5pt}
{5pt}
{\slshape}
{\parindent}
{\bfseries}
{.}
{.5em}
{}

\theoremstyle{lause}
\newtheorem{theorem}{Theorem}[section]
\newtheorem{lemma}[theorem]{Lemma}

\newtheorem{problem}[theorem]{Problem}

\newtheoremstyle{maaritelma}
{5pt}
{5pt}
{\rmfamily}
{\parindent}
{\bfseries}
{.}
{.5em}
{}

\theoremstyle{maaritelma}
\newtheorem{definition}[theorem]{Definition}

\newtheorem{example}[theorem]{Example}
\newtheorem{remark}[theorem]{Remark}

\numberwithin{equation}{section}

\begin{document}

\thispagestyle{empty}

\begin{center}

{\large{\textbf{On the theory of balayage on locally compact spaces}}}

\vspace{18pt}

\textbf{Natalia Zorii}

\vspace{18pt}

\emph{Dedicated to Professor Wolfgang L.\ Wendland on the occasion of his 85th birthday}\vspace{8pt}

\footnotesize{\address{Institute of Mathematics, Academy of Sciences
of Ukraine, Tereshchenkivska~3, 01601,
Kyiv-4, Ukraine\\
natalia.zorii@gmail.com }}

\end{center}

\vspace{12pt}

{\footnotesize{\textbf{Abstract.} The paper deals with the theory of balayage of Radon measures $\mu$ of finite energy on a locally compact space $X$ with respect to a consistent kernel $\kappa$ satisfying the domination principle. Such theory is now specified for the case where the topology on $X$ has a countable base, while any $f\in C_0(X)$, a continuous function on $X$ of compact support, can be approximated in the inductive limit topology on the space $C_0(X)$ by potentials $\kappa\lambda:=\int\kappa(\cdot,y)\,d\lambda(y)$ of measures $\lambda$ of finite energy. In particular, we show that then the inner balayage can always be reduced to balayage to Borel sets. In more details, for arbitrary $A\subset X$, there exists a $K_\sigma$-set $A_0\subset A$ such that \[\mu^A=\mu^{A_0}=\mu^{*A_0}\text{ \ for all $\mu$,}\]
$\mu^A$ and $\mu^{*A}$ denoting the inner and the outer balayage of $\mu$ to $A$, respectively. Furthermore, $\mu^A$ is now uniquely determined by the symmetry relation $\int\kappa\mu^A\,d\lambda=\int\kappa\lambda^A\,d\mu$, $\lambda$ ranging over a certain countable family of measures depending on $X$ and $\kappa$ only.
As an application of these theorems, we analyze the convergence of inner and outer swept measures and their potentials.
The results obtained do hold for many interesting kernels in classical and modern potential theory on $\mathbb R^n$, $n\geqslant2$.}}
\symbolfootnote[0]{\quad 2010 Mathematics Subject Classification: Primary 31C15.}
\symbolfootnote[0]{\quad Key words: Radon measures on a locally compact space, consistent kernel, potential, energy, domination principle, capacitability, inner and outer balayage.}

\vspace{6pt}

\markboth{\emph{Natalia Zorii}} {\emph{On the theory of balayage on locally compact spaces}}

\section{Introduction. Main results}\label{sec-intr} This paper deals with a theory of balayage of (real-valued Radon) measures of finite energy on a locally compact (Hausdorff) space $X$ in the setting of potentials with respect to a symmetric, lower semicontinuous (l.s.c.)\ {\it kernel\/} $\kappa:X\times X\to[0,\infty]$. Such a theory has recently been developed in the author's work \cite{Z-arx1} (see Sect.~\ref{i1} for a brief exposition of \cite{Z-arx1}). Being mainly concerned with the study of inner balayage, the results established in \cite{Z-arx1} are not covered by the balayage theory in the setting of balayage spaces \cite{BH} or $H$-cones \cite{BBC}, or the theory developed in \cite{Fu5}, the last being based on the investigation of outer balayage to quasiclosed sets.

The theory established in \cite{Z-arx1} generalizes H.~Cartan's theory of Newtonian inner balayage \cite{Ca2} to a kernel $\kappa$ on $X$ satisfying the energy, consistency, and domination principles (for definitions, see Sect.~\ref{sec}).
{\it Throughout the present paper, unless explicitly stated otherwise, we shall always assume these three principles to hold.}

Denote by $\mathfrak M=\mathfrak M(X)$ the linear space of all (signed) measures $\nu$ on $X$, equipped with the {\it vague\/} ($=${\it weak\,$^*$}) topology of pointwise convergence on the class $C_0=C_0(X)$ of all continuous functions $f:X\to\mathbb R$ of compact support, and by $\mathfrak M^+=\mathfrak M^+(X)$ the cone of all positive $\nu\in\mathfrak M$. The kernel $\kappa$ being strictly positive definite, all $\nu\in\mathfrak M$ of finite {\it energy\/} $\kappa(\nu,\nu):=\int\kappa\,d(\nu\otimes\nu)$ form a pre-Hil\-bert space $\mathcal E=\mathcal E(X)$ with the inner product $(\mu,\nu):=\kappa(\mu,\nu):=\int\kappa\,d(\mu\otimes\nu)$. The (strictly positive definite) kernel $\kappa$ being consistent, hence perfect \cite{F1}, the cone $\mathcal E^+=\mathcal E^+(X):=\mathcal E\cap\mathfrak M^+$ is complete in the {\it strong\/} topology, determined by the energy norm $\|\nu\|:=\sqrt{\kappa(\nu,\nu)}$, and the strong topology on $\mathcal E^+$ is finer than the vague topology (see Sect.~\ref{sec} for more details).

\subsection{Basic facts of the theory of inner balayage}\label{i1}
Given a measure $\mu\in\mathcal E^+$ and a set $A\subset X$, let $\Lambda_{A,\mu}$ stand for the class of all $\nu\in\mathcal E^+$ having the property
\begin{equation}\label{io}\kappa\nu\geqslant\kappa\mu\text{ \ n.e.\ on $A$},\end{equation}
where $\kappa\nu:=\int\kappa(\cdot,y)\,d\nu(y)$ is the {\it potential}, and {\it n.e.}\ ({\it nearly everywhere\/}) means that the inequality holds on all of $A$ except for a subset  of $A$ of inner capacity zero.

\begin{definition}[{\rm see \cite[Definition~4.1]{Z-arx1}}]\label{i-b-def}The {\it inner balayage\/} $\mu^A$ of $\mu\in\mathcal E^+$ to $A\subset X$ is the (unique) measure in $\Lambda_{A,\mu}$ of minimal potential, that is, $\mu^A\in\Lambda_{A,\mu}$ and
\[\kappa\mu^A=\min_{\nu\in\Lambda_{A,\mu}}\,\kappa\nu\text{ \ everywhere on $X$}.\]\end{definition}

Observe that this definition is in agreement with the concept of inner Newtonian balayage by Cartan as well as with that of inner Riesz balayage by the author (see \cite[Section~19, Theorem~1]{Ca2} and \cite[Theorem~4.3]{Z-bal}, respectively).

The uniqueness of the inner balayage $\mu^A$ follows easily by the energy principle, cf.\ (\ref{en}). As for the existence of $\mu^A$, suppose for a moment that $A=K$ is compact. Based on the consistency and domination principles, one can prove by generalizing the (classical) Gauss variational method\footnote{See \cite{C0,Ca2}, cf.\ also \cite[Section~IV.5.23]{L}.} that the inner balayage $\mu^K\in\mathcal E^+$ does exist; it is concentrated on the set $K$ itself (that is, $\mu^K\in\mathcal E^+_K$), and it is uniquely characterized within $\mathcal E^+_K$ by the equality
$\kappa\mu^K=\kappa\mu$ n.e.\ on $K$.
However, if the requirement of $A$ being compact is now dropped, then, in general, neither of the last two conclusions remains in force (unless of course $A$ is closed); compare with Theorem~\ref{th-intr}\,(a).

Denote by $\mathcal E^+_A$ the set of all $\nu\in\mathcal E^+$ concentrated on $A\subset X$, and by $\mathcal E'_A$ the closure of $\mathcal E^+_A$ in the strong topology on the pre-Hil\-bert space $\mathcal E$.

\begin{theorem}[{\rm see \cite[Theorem~4.3]{Z-arx1}}]\label{th-intr}For any\/ $\mu\in\mathcal E^+$ and\/ $A\subset X$, the inner balayage\/ $\mu^A$, introduced by Definition\/~{\rm\ref{i-b-def}}, does exist; it satisfies the relations
\begin{align}
\label{eq-pr-10}\kappa\mu^A&=\kappa\mu\text{ \ n.e.\ on\ }A,\\
\notag\kappa\mu^A&=\kappa\mu\text{ \ $\mu^A$-a.e.,}\\
\notag\kappa\mu^A&\leqslant\kappa\mu\text{ \ on $X$.}
\end{align}
Furthermore, $\mu^A$  belongs to\/ $\mathcal E'_A$, and it can equivalently be determined by means of either of the following two characteristic properties:
\begin{itemize}\item[\rm(a)] $\mu^A$ is the unique measure in\/ $\mathcal E'_A$ satisfying\/ {\rm(\ref{eq-pr-10})}.
\item[\rm(b)] $\mu^A$ is the unique measure in\/ $\mathcal E'_A$ such that\/\footnote{The inner balayage $\mu^A$ can therefore be found as the {\it orthogonal projection\/} of $\mu\in\mathcal E^+$ in the pre-Hil\-bert space $\mathcal E$ onto the convex cone $\mathcal E_A'\subset\mathcal E^+$, the strong closure of $\mathcal E^+_A$. Observe that the cone $\mathcal E_A'$ is strongly complete, being a strongly closed subset of the strongly complete cone $\mathcal E^+$, and hence the orthogonal projection of $\mu$ onto $\mathcal E_A'$ does exist, cf.\ \cite[Theorem~1.12.3]{E2}.}
\[\|\mu-\mu^A\|=\min_{\nu\in\mathcal E_A'}\,\|\mu-\nu\|=\inf_{\nu\in\mathcal E_A^+}\,\|\mu-\nu\|.\]
\end{itemize}
\end{theorem}

Given $A\subset X$, denote by $\mathfrak C_A$ the upward directed set of all compact subsets $K$ of $A$, where $K_1\leqslant K_2$ if and only if $K_1\subset K_2$. If $T$ is a topological Hausdorff space, and a net $(t_K)_{K\in\mathfrak C_A}\subset T$ converges to $t_0\in T$, then we shall indicate this fact by writing
\[t_K\to t_0\text{ \ in $T$ as $K\uparrow A$}.\]

The following convergence theorem justifies the term `inner balayage'.

\begin{theorem}[{\rm see \cite[Theorem~4.8]{Z-arx1}}]\label{th-intr''}For any\/ $\mu\in\mathcal E^+$ and\/ $A\subset X$, the following two assertions on the convergence of inner swept measures and their potentials hold:
\begin{itemize}\item[\rm(a)] $\mu^K\to\mu^A$ strongly and vaguely in\/ $\mathcal E^+$ as\/ $K\uparrow A$.
\item[\rm(b)] $\kappa\mu^K\uparrow\kappa\mu^A$ pointwise on\/ $X$ as\/ $K\uparrow A$,
and hence
\begin{equation}\label{later-bal1}\kappa\mu^A=\sup_{K\in\mathfrak C_A}\,\kappa\mu^K\text{ \ everywhere on $X$.}\end{equation}
\end{itemize}\end{theorem}

\subsection{About the results obtained}\label{about} In the rest of the present paper, we shall often require assumptions $(\mathcal A_1)$ and/or $(\mathcal A_2)$ to be satisfied, where:
\begin{itemize}\item[$(\mathcal A_1)$] {\it The topology on a locally compact space\/ $X$ has a countable base.}
\item[$(\mathcal A_2)$] {\it The set of all\/ $f\in C_0(X)$ representable as potentials\/ $\kappa\lambda$ with\/ $\lambda\in\mathcal E(X)$ is dense in the space\/ $C_0(X)$ equipped with the inductive limit topology\/}(\footnote{Regarding the inductive limit topology on the space $C_0(X)$, see Bourbaki \cite[Section~II.4.4]{B4} and \cite[Section~III.1.1]{B2} (cf.\ also Sect.~\ref{sec-aux} below).}),(\footnote{We shall now provide three examples where $(\mathcal A_2)$ holds, whereas the validity of $(\mathcal A_1)$ is obvious.

    In the first, $\kappa$ is the $\alpha$-Riesz kernel $|x-y|^{\alpha-n}$, $\alpha\in(0,n)$, on $X:=\mathbb R^n$, $n\geqslant2$. Observe that for every $f\in C_0(\mathbb R^n)$, there exist a compact set $K\subset\mathbb R^n$ and a sequence $(f_j)\subset C_0^\infty(\mathbb R^n)$ (obtained by regularization \cite[p.~22]{S}) such that all the $f$ and $f_j$ equal $0$ on $K^c:=\mathbb R^n\setminus K$, and moreover $(f_j)$ converges to $f$ uniformly on $K$ (hence, also in the inductive limit topology on $C_0(\mathbb R^n)$, cf.\ Lemma~\ref{foot-conv}). Since each $\varphi\in C_0^\infty(\mathbb R^n)$ can be represented as the $\alpha$-Riesz potential of a measure on $\mathbb R^n$ of finite $\alpha$-Riesz energy (see \cite[Lemma~1.1]{L} and \cite[Lemma~3.4]{Zor21}), $(\mathcal A_2)$ indeed holds.

    In the next two examples, $X:=D$, where $D\subset\mathbb R^n$, $n\geqslant2$, is open. Then $(\mathcal A_2)$ holds if either $\kappa$ is the ($2$-)Green kernel for the Laplace operator on Greenian $D$, or $\kappa$ is the $\alpha$-Green kernel $g^\alpha_D$ of order $\alpha\in(1,2)$ for the fractional Laplacian on bounded $D$ of class $C^{1,1}$. This follows by applying \cite[p.~75, Remark]{L}, resp.\ \cite[Eq.~(19)]{Bogdan}, to $\varphi\in C_0^\infty(D)$, and then utilizing the same approximation technique as above. (Recall that $D$ is said to be of class $C^{1,1}$ if for every $y\in\partial D$, there exist open balls $B(x,r)\subset D$ and $B(x',r)\subset D^c$, where $r>0$, that are tangent at $y$, see \cite[p.~458]{Bogdan}. Regarding the definition of the $\alpha$-Green kernel $g^\alpha_D$, see \cite[Section~IV.5.20]{L}, cf.\ \cite[Section~2]{Bogdan}.)\label{FF}}).
\end{itemize}

The main result of this work, formulated as Theorem~\ref{th-m} below, basically shows that under these two assumptions
the inner balayage can always be reduced to balayage to Borel sets.

\begin{theorem}\label{th-m}Suppose that\/ $(\mathcal A_1)$ and\/ $(\mathcal A_2)$ both hold. For arbitrary\/ $A\subset X$, there exists a\/ $K_\sigma$-set\,\footnote{A set $Q\subset X$ is said to be $\sigma$-{\it com\-pact\/} or {\it of class\/ $K_\sigma$} (in short, a $K_\sigma$-{\it set\/}) if it is representable as a countable union of compact subsets of $X$. A locally compact space $X$ of class $K_\sigma$ is also said to be {\it countable at infinity}, cf.\ \cite[Section~I.9, Definition~5]{B1}.} $A_0\subset A$ having the property
\begin{equation}\label{eq-main1}\mu^A=\mu^{A_0}\text{ \ for all\/ $\mu\in\mathcal E^+$.}\end{equation}
Moreover,
\begin{equation}\label{eq-main1''}\mu^A=\mu^{A_0}=\mu^{*A_0}\text{ \ for all\/ $\mu\in\mathcal E^+$},\end{equation}
where\/ $\mu^{*A_0}$ denotes the outer balayage of\/ $\mu$ to\/ $A_0$ {\rm(see \cite[Definition~9.1]{Z-arx1})}.
\end{theorem}

\begin{remark}\label{def-outer}The concept of\/ {\it outer balayage\/} was defined in \cite{Z-arx1} similarly to that of inner balayage, the only difference being in replacing an exceptional set in (\ref{io}) by that of {\it outer capacity\/} zero. As seen from \cite[Section~19, Theorem~1$'$]{Ca2}, this concept
agrees with that of outer Newtonian balayage by Cartan. Furthermore, the outer balayage is {\it unique}, which follows easily by the energy principle, cf.\ (\ref{en}). As for the existence of the outer balayage, see Theorem~\ref{l-outer} below.\end{remark}

Note that the $K_\sigma$-set $A_0\subset A$, introduced in Theorem~\ref{th-m}, does not depend on the choice of $\mu\in\mathcal E^+$. Actually, we have proved even more (see Theorem~\ref{cor-m}): relation (\ref{eq-main1}) remains valid with $A_0$ replaced by {\it any\/} $Q$ such that $A_0\subset Q\subset A$, while (\ref{eq-main1''}) still holds if $A_0$ is replaced by {\it any Borel\/} $B$, where $A_0\subset B\subset A$.

\begin{remark}\label{rem-sign}Theorems~\ref{th-m} and \ref{cor-m} can certainly be extended to {\it signed\/} measures $\mu$ of finite energy by defining the inner and the outer balayage by linearity:
\begin{equation}\label{def-sign}\mu^A:=(\mu^+)^A-(\mu^-)^A,\quad\mu^{*A}:=(\mu^+)^{*A}-(\mu^-)^{*A},\end{equation}
$\mu^+$ and $\mu^-$ denoting the positive and negative parts of $\mu$ in the Hahn--Jor\-dan decomposition. (Here we have used the fact that $\mu$ has finite energy if and only if so do both $\mu^+$ and $\mu^-$, the kernel $\kappa$ being positive definite.)\end{remark}

It is worth emphasizing that for any $\mu\in\mathcal E$ and $A\subset X$, the inner balayage $\mu^A$ is uniquely determined within $\mathcal E$ by the following symmetry relation (see Theorem~\ref{l-alt}):
\[\kappa(\mu^A,\nu)=\kappa(\nu^A,\mu)\text{ \ for all $\nu\in\mathcal E$}.\]
But if now $(\mathcal A_1)$ and $(\mathcal A_2)$ are both met, then this characteristic property of the inner balayage $\mu^A$ needs only to be verified for certain {\it countably many\/} $\nu\in\mathcal E$ that are independent of the choice of $A$ and $\mu$.
In more details, the following theorem holds.

\begin{theorem}\label{l-alt-countt}Under assumptions\/ $(\mathcal A_1)$ and\/ $(\mathcal A_2)$, there exists a countable set\/ $\mathcal E^\circ\subset\mathcal E$ such that, for any\/ $\mu\in\mathcal E$ and $A\subset X$, the inner balayage\/ $\mu^A$ is uniquely determined within\/ $\mathcal E$ by the relation
\begin{equation}\label{rel1}\kappa(\mu^A,\lambda)=\kappa(\lambda^A,\mu)\text{ \ for all $\lambda\in\mathcal E^\circ$}.\end{equation}
That is, if for some\/ $\xi\in\mathcal E$,
\begin{equation}\label{rel2}\kappa(\xi,\lambda)=\kappa(\lambda^A,\mu)\text{ \ for all $\lambda\in\mathcal E^\circ$,}\end{equation}
then
\[\xi=\mu^A.\]
If\/ $A=B$ is Borel, then the same remains valid with\/ $\mu^A$ and\/ $\lambda^A$ replaced by the outer swept measures\/ $\mu^{*B}$ and\/ $\lambda^{*B}$, respectively.
\end{theorem}

The proofs of Theorems~\ref{th-m} and \ref{l-alt-countt} (Sect.~\ref{sec-last-1}) use substantially the following Theorem~\ref{l-count}, whereas the proof of Theorem~\ref{l-count} (Sect.~\ref{count-proof}) is based on Lemmas~\ref{foot-conv} and \ref{foot-conv'} (Sect.~\ref{sec-aux}), specifying the concept of the inductive limit topology on $C_0(X)$ for $X$ satisfying the second axiom of countability (more generally, for $X$ of class $K_\sigma$).

\begin{theorem}\label{l-count}Dropping for a moment the requirement of the energy, consistency, and domination principles, assume that\/ $(\mathcal A_1)$ and\/ $(\mathcal A_2)$ both hold. Then there exists a countable set\/ $\mathcal E^\circ\subset\mathcal E$ having the following property: a sequence\/ $(\nu_k)\subset\mathfrak M^+$ converges vaguely to\/ $\nu_0$ if and only if\/\footnote{Both (\ref{l-count1}) and (\ref{l-count2}) are understood in the sense that the integrals involved therein do exist.}
\begin{equation}\label{l-count1}\lim_{k\to\infty}\,\int\kappa\lambda\,d\nu_k=\int\kappa\lambda\,d\nu_0\text{ \ for all $\lambda\in\mathcal E^\circ$.}\end{equation}
The vague topology on\/ $\mathfrak M$ being Hausdorff, any two\/ $\mu,\nu\in\mathfrak M$ are thus equal if\/ {\rm(}and only if\/{\rm)}
\begin{equation}\label{l-count2}\int\kappa\lambda\,d\mu=\int\kappa\lambda\,d\nu\text{ \ for all $\lambda\in\mathcal E^\circ$.}\end{equation}
\end{theorem}

As an application of the above-quoted theorems, we establish the convergence of inner and outer swept measures and their potentials under the exhaustion of arbitrary $A\subset X$ by the sequence $(A\cap U_j)$, $(U_j)$ being an increasing sequence of Borel sets with the union $X$ (see Theorem~\ref{pr-cont}, compare with Theorem~\ref{th-intr''}).

The theory thereby developed covers the case of the Green kernel for the Laplace operator on a Greenian set $D\subset\mathbb R^n$, $n\geqslant2$ (thus in particular that of the Newtonian kernel $|x-y|^{2-n}$ on $\mathbb R^n$, $n\geqslant3$), that of the $\alpha$-Riesz kernels $|x-y|^{\alpha-n}$ on $\mathbb R^n$, $n\geqslant2$, where $\alpha\leqslant2$ and $\alpha<n$, as well as  that of the $\alpha$-Green kernels $g^\alpha_D$ of order $\alpha\in(1,2)$ for the fractional Laplacian on a bounded open $C^{1,1}$-set $D\subset\mathbb R^n$, $n\geqslant2$ (cf.\ footnote~\ref{FF} and Example~\ref{rem:clas}).
This suggests that the present work may be useful, for instance, in
the analytic, constructive, and numerical analysis of  minimum energy problems. See e.g.\ recent papers \cite{DFHSZ2}, \cite{Dr0}, \cite{FZ-Pot2}, \cite{HWZ}, \cite{Z-Pot}--\cite{ZPot3}, \cite{Z-AMP} for some applications of balayage to minimum energy problems in various settings, \cite{Z-Pot}--\cite{ZPot3} dealing with a general function kernel $\kappa$ on a locally compact space $X$.

\section{Basic facts of potential theory on locally compact spaces}\label{sec}

The present section review some basic facts of the theory of potentials on a locally compact space $X$, to be used throughout the paper.
Here and in the sequel we follow the notation and conventions of Sect.~\ref{sec-intr}.

\begin{lemma}[{\rm see \cite[Section~IV.1, Proposition~4]{B2}}]\label{lemma-semi}For any l.s.c.\ function\/ $\psi:X\to[0,\infty]$, the mapping\/ $\mu\mapsto\int\psi\,d\mu$ is
vaguely l.s.c.\ on\/ $\mathfrak M^+$, the integral here being understood as upper integral.\end{lemma}

Given a (positive, symmetric, l.s.c.)\ kernel $\kappa$ on $X$ and (signed) measures $\mu,\nu\in\mathfrak M$, define the {\it potential\/} and the {\it mutual energy\/}  by
\begin{align*}\kappa\mu(x)&:=\int\kappa(x,y)\,d\mu(y),\quad x\in X,\\
\kappa(\mu,\nu)&:=\int\kappa(x,y)\,d(\mu\otimes\nu)(x,y),
\end{align*}
respectively, provided the right-hand side is well defined as a finite number or $\pm\infty$ (for more details, see \cite[Section~2.1]{F1}). In particular, if $\mu,\nu$ are positive, then $\kappa\mu(x)$, resp.\ $\kappa(\mu,\nu)$, is well defined and represents a positive l.s.c.\ function of $(x,\mu)\in X\times\mathfrak M^+$, resp.\ $(\mu,\nu)\in\mathfrak M^+\times\mathfrak M^+$ (the {\it principle of descent\/} \cite[Lemma~2.2.1]{F1}, cf.\ Lemma~\ref{lemma-semi}).
For $\mu=\nu$, the mutual energy
$\kappa(\mu,\nu)$ defines the {\it energy\/} $\kappa(\mu,\mu)$ of $\mu$.

In what follows, except for Sect.~\ref{proof-prop}, we assume that a kernel $\kappa$ is strictly positive definite and satisfies the
consistency and domination principles. Recall that a kernel $\kappa$ is said to be {\it positive definite\/} if for any $\mu\in\mathfrak M$, the energy $\kappa(\mu,\mu)$ is ${}\geqslant0$ whenever defined, and {\it strictly positive definite\/}\footnote{For such a kernel $\kappa$, the {\it energy principle\/} is said to hold.} if, moreover, $\kappa(\mu,\mu)=0$ implies $\mu=0$.
All $\mu\in\mathfrak M$ with finite energy then form a pre-Hil\-bert space $\mathcal E=\mathcal E(X)$ with the inner product $(\mu,\nu):=\kappa(\mu,\nu)$ and the norm $\|\mu\|:=\sqrt{\kappa(\mu,\mu)}$, see \cite[Section~3.1]{F1}. The (Hausdorff) topology on $\mathcal E$ defined by means of the norm $\|\cdot\|$ is said to be {\it strong}.

It follows by use of the Cauchy--Schwarz (Bunyakovski) inequality\footnote{Another application of (\ref{Ca}) shows that any positive definite kernel $\kappa$ satisfying assumption $(\mathcal A_2)$ (Sect.~\ref{about}) must be strictly positive definite. Indeed, if $\kappa(\mu,\mu)=0$ for some $\mu\in\mathcal E$, then, by (\ref{Ca}),
\[|\mu(f)|=|\kappa(\lambda,\mu)|\leqslant\|\lambda\|\cdot\|\mu\|=0\]
for all $f\in C_0$ representable as potentials $\kappa\lambda$ with $\lambda\in\mathcal E$. The set $C_0\cap\{\kappa\lambda: \lambda\in\mathcal E\}$ being dense in the space $C_0$ equipped with the inductive limit topology, applying \cite[Section~III.1.7]{B2} gives $\mu=0$.}
\begin{equation}\label{Ca}|\kappa(\mu,\nu)|\leqslant\|\mu\|\cdot\|\nu\|,\text{ \ where $\mu,\nu\in\mathcal E$},\end{equation} that $\mu\in\mathfrak M$ has finite energy if (and only if) so do both $\mu^+,\mu^-$. Thus $\mathcal E=\mathcal E^+-\mathcal E^+$, where $\mathcal E^+=\mathcal E^+(X):=\mathcal E\cap\mathfrak M^+$.

A (strictly positive definite) kernel $\kappa$ is called {\it consistent\/} \cite[Section~3.3]{F1} if every strong Cauchy net in the cone $\mathcal E^+$ converges strongly to any of its vague limit points.
The cone $\mathcal E^+$ then becomes {\it strongly complete}, a strongly bounded part of $\mathcal E^+$ being vaguely relatively compact by \cite[Lemma~2.5.1]{F1}. As the strong limit of a strong Cauchy net $(\mu_k)\subset\mathcal E^+$ is unique, all the vague limit points of such $(\mu_k)$ must be equal. The vague topology on $\mathfrak M$ being Hausdorff, we thus conclude by applying \cite[Section~I.9.1, Corollary]{B1} that any strong Cauchy net in $\mathcal E^+$ converges to the same (unique) limit both strongly and vaguely. In Fuglede's terminology \cite[Section~3.3]{F1}, a strictly positive definite, consistent kernel is, therefore, {\it perfect}.

Given a set $A\subset X$, denote by $\mathfrak M^+_A$ the cone of all $\mu\in\mathfrak M^+$ {\it concentrated on\/}
$A$, which means that $A^c:=X\setminus A$ is locally $\mu$-neg\-lig\-ible, or equivalently, that $A$ is $\mu$-meas\-ur\-able and $\mu=\mu|_A$, $\mu|_A:=1_A\cdot\mu$ being the {\it trace\/} of $\mu$ to $A$ \cite[Section~IV.14.7]{E2}. (Note that for $\mu\in\mathfrak M^+_A$, the indicator function $1_A$ of $A$ is locally $\mu$-int\-egr\-able.) The total mass of $\mu\in\mathfrak M^+_A$ is $\mu(X)=\mu_*(A)$, $\mu_*(A)$ and $\mu^*(A)$ denoting the {\it inner\/} and the {\it outer\/} measure of $A$, respectively. If, moreover, $A$ is closed, or if $A^c$ is contained in a countable union of sets $Q_k$ with $\mu^*(Q_k)<\infty$,\footnote{If the latter holds, $A^c$ is said to be $\mu$-$\sigma$-{\it finite\/} \cite[Section~IV.7.3]{E2}. This occurs e.g.\ if the measure $\mu$ is {\it bounded\/} (that is, with $\mu(X)<\infty$), or if the space $X$ is of class $K_\sigma$.} then for any $\mu\in\mathfrak M^+_A$, $A^c$ is $\mu$-neg\-lig\-ible, i.e.\ $\mu^*(A^c)=0$. In particular, if $A$ is closed, then $\mathfrak M^+_A$ consists of all $\mu\in\mathfrak M^+$ {\it supported by\/} $A$ (that is, with support contained in $A$, cf.\ \cite[Section~III.2.2]{B2}).

Denoting now $\mathcal E^+_A:=\mathfrak M^+_A\cap\mathcal E$, we define the {\it inner capacity\/} of $A$ by
\[c_*(A):=\Bigl[\inf_{\mu\in\mathcal E^+_A: \ \mu(X)=1}\,\kappa(\mu,\mu)\Bigr]^{-1}.\]
(The infimum over the empty set is interpreted as $+\infty$.)
Then \cite[p.~153, Eq.~(2)]{F1}
\begin{equation*}\label{153}c_*(A)=\sup\,c_*(K)\text{ \ ($K\subset A$ compact)}.\end{equation*}
Also, by homogeneity reasons (cf.\ \cite[Lemma~2.3.1]{F1}),
\begin{equation*}c_*(A)=0\iff\mathcal E^+_A=\{0\}\iff\mathcal E^+_K=\{0\}\text{ \ for every compact $K\subset A$}.\label{2.3.1}\end{equation*}

We are thus led to the following conclusion.

\begin{lemma}\label{l-negl}Given\/ $\mu\in\mathcal E^+$, consider a\/ $\mu$-measurable set\/ $A\subset X$ with\/ $c_*(A)=0$. Then\/ $A$ is locally\/ $\mu$-neg\-lig\-ible; and\/ $A$ is\/ $\mu$-neg\-lig\-ible if, moreover, it is\/ $\mu$-$\sigma$-fi\-ni\-te.\end{lemma}

Defining further the {\it outer capacity\/} of a set $A$ by
\[c^*(A):=\inf\,c_*(D),\]
$D$ ranging over all open sets that contain $A$, we obtain $c_*(A)\leqslant c^*(A)$.
A set $A$ is said to be {\it capacitable\/} if $c_*(A)=c^*(A)$.
The following result on capacitability is established by a direct application of \cite[Theorem~4.5]{F1}.

\begin{theorem}\label{l-top}Any Borel subset of a locally compact perfectly normal space\/ $X$ of class\/ $K_\sigma$, endowed with a perfect kernel\/ $\kappa$, is capacitable.\footnote{By Urysohn's theorem \cite[Section~IX.1, Theorem~1]{B3}, a topological Hausdorff space $Y$ is said to be {\it normal\/} if for any two disjoint closed subsets $F_1,F_2$ of $Y$, there exist disjoint open sets $D_1,D_2$ such that $F_i\subset D_i$ $(i=1,2)$. Further, a normal space $Y$ is said to be {\it perfectly normal\/} \cite[Section~IX.4, Exercise~7]{B3} if each closed subset of $Y$ is a countable intersection of open sets (or, equivalently, if each open subset of $Y$ is a countable union of closed sets).}
\end{theorem}

An assertion $\mathcal U(x)$ is said to hold {\it nearly everywhere\/} ({\it n.e.})\ on a set $A$ if the set $E$ of all $x\in A$ for which $\mathcal U(x)$ fails has inner capacity zero: $c_*(E)=0$. Replacing here $c_*(\cdot)$ by $c^*(\cdot)$ leads to the notion {\it qua\-si-ev\-ery\-whe\-re\/} ({\it q.e.})\ on a set $A$.

For any $\mu\in\mathcal E$, the potential $\kappa\mu$ is well defined and finite q.e.\ (hence, n.e.) on $X$, see \cite{F1} (Corollary to Lemma~3.2.3). Furthermore, it follows from \cite[Lemma~3.2.1]{F1} by use of the energy principle that
\begin{equation}\label{en}\mu=0\iff\kappa\mu=0\text{ \ n.e.\ on $X$}\iff\kappa\mu=0\text{ \ q.e.\ on $X$}.\end{equation}

We finally recall that a kernel $\kappa$ is said to satisfy {\it the domination principle\/} if for any $\mu,\nu\in\mathcal E^+$ such that $\kappa\mu\leqslant\kappa\nu$ $\mu$-a.e., the same inequality holds on all of $X$.

\begin{example}\label{rem:clas} The $\alpha$-Riesz kernel $|x-y|^{\alpha-n}$ on $\mathbb R^n$, $n\geqslant2$, with $\alpha\leqslant2$ and $\alpha<n$ (thus in particular the Newtonian kernel $|x-y|^{2-n}$ on $\mathbb R^n$, $n\geqslant3$) satisfies the energy, consistency, and domination principles \cite[Theorems~1.15, 1.18, 1.27, 1.29]{L}, and so does the associated $\alpha$-Green kernel $g^\alpha_D$ on any open subset of $\mathbb R^n$ \cite[Theorems~4.6, 4.9, 4.11]{FZ}. The ($2$-)Green kernel on a planar Greenian set is strictly positive definite \cite[Section~I.XIII.7]{Doob} and consistent \cite{E}, and it satisfies the domination principle (see \cite[Theorem~5.1.11]{AG} or \cite[Section~I.V.10]{Doob}). Finally, the logarithmic kernel $-\log\,|x-y|$ on a closed disc in $\mathbb R^2$ of radius ${}<1$ is strictly positive definite and satisfies Frostman's maximum principle \cite[Theorems~1.6, 1.16]{L}, and hence it is perfect \cite[Theorem~3.4.2]{F1}. However, the domination principle then fails in general; it does hold only in a weaker sense where the measures $\mu,\nu$ involved in the ab\-ove-quo\-ted definition meet the additional requirement  $\nu(\mathbb R^2)\leqslant\mu(\mathbb R^2)$ \cite[Theorem~3.2]{ST}.\footnote{Because of this obstacle, the theory of balayage on a locally compact space $X$, developed in \cite{Z-arx1} and the present paper, does not cover the case of the logarithmic kernel on $\mathbb R^2$.}\end{example}

\section{Further properties of balayage} In all that follows, we keep the notation and conventions of Sects.~\ref{sec-intr} and~\ref{sec}.

The present section provides some further properties of $\mu^A$ and $\mu^{*A}$, the inner and the outer balayage of a (signed) measure $\mu$ of finite energy on a locally compact space $X$ to a set $A\subset X$ (for definitions, see Definition~\ref{i-b-def} and Remarks~\ref{def-outer} and \ref{rem-sign}).

\begin{lemma}[{\rm Monotonicity property}]\label{cor-mon}For any\/ $\mu\in\mathcal E^+$,
\begin{equation}\label{mon}\kappa\mu^A\leqslant\kappa\mu^Q\text{ \ whenever $A\subset Q$}.\end{equation}
\end{lemma}

\begin{proof}This follows directly from (\ref{later-bal1}) by noting that $\mathfrak C_A\subset\mathfrak C_Q$.\end{proof}

\begin{theorem}\label{l-alt}
For any\/ $\mu\in\mathcal E$ and\/ $A\subset X$,
\begin{equation}\label{alternative}\kappa(\mu^A,\nu)=\kappa(\mu,\nu^A)\text{ \ for all\/ $\nu\in\mathcal E$}.\end{equation}
Furthermore, symmetry relation\/ {\rm(\ref{alternative})} determines the inner balayage\/ $\mu^A$ uniquely within\/ $\mathcal E$; that is, if for some\/ $\xi\in\mathcal E$,
$\kappa(\xi,\nu)=\kappa(\mu,\nu^A)$ for all\/ $\nu\in\mathcal E$, then\/ $\xi=\mu^A$.
\end{theorem}

\begin{proof}On account of definition (\ref{def-sign}), we can certainly assume that $\mu,\nu$ are positive, the mutual energy being bilinear on measures of the class $\mathcal E$.

Consider first $K\in\mathfrak C_A$. Denoting $E:=K\cap\{\kappa\mu^K\ne\kappa\mu\}$, we obtain $c_*(E)=0$ by (\ref{eq-pr-10}). But the set $E$ is universally measurable, while $\nu^K\in\mathcal E^+$ is bounded and supported by $K$. By use of Lemma~\ref{l-negl}, we therefore get $\kappa\mu^K=\kappa\mu$ $\nu^K$-a.e., whence
\[\kappa(\mu^K-\mu,\nu^K)=0.\]
Similarly, $\kappa(\nu^K-\nu,\mu^K)=0$, which gives by subtraction
\begin{equation}\label{K}\kappa(\mu^K,\nu)=\kappa(\mu,\nu^K).\end{equation}

Since the potentials $\kappa\mu^K$ increase pointwise on $X$ as $K$ ranges through the directed set $\mathfrak C_A$ and do not exceed $\kappa\mu^A$ (Theorem~\ref{th-intr''}(b)),
\[\limsup_{K\uparrow A}\,\kappa(\mu^K,\nu)\leqslant\kappa(\mu^A,\nu).\] On the other hand,
$\mu^K\otimes\nu\to\mu^A\otimes\nu$ vaguely in $\mathfrak M^+\times\mathfrak M^+$ as $K\uparrow A$,
which is implied by the vague convergence of the net $(\mu^K)_{K\in\mathfrak C_A}$ to $\mu^A$ (Theorem~\ref{th-intr''}(a)) with the aid of \cite[Section~III.4, Exercise~1b]{B2}. Thus, by the principle of descent (cf.\ Lemma~\ref{lemma-semi}),
\[\kappa(\mu^A,\nu)\leqslant\liminf_{K\uparrow A}\,\kappa(\mu^K,\nu),\]
which together with the preceding display yields
\[\kappa(\mu^A,\nu)=\lim_{K\uparrow A}\,\kappa(\mu^K,\nu).\]
As the same holds with $\mu$ and $\nu$ interchanged, letting $K\uparrow A$ in (\ref{K}) establishes (\ref{alternative}) for any positive (hence, also for any signed) measures $\mu,\nu$ of finite energy.

To prove
the latter part of the theorem, suppose now that for some $\xi\in\mathcal E$, $\kappa(\xi,\nu)=\kappa(\mu,\nu^A)$ for all $\nu\in\mathcal E$. Subtracting this equality from (\ref{alternative}) gives
\[\kappa(\xi-\mu^A,\nu)=0\text{ \ for all $\nu\in\mathcal E$},\]
hence $\kappa(\xi-\mu^A,\xi-\mu^A)=0$ by substituting $\nu:=\xi-\mu^A$, and consequently $\xi=\mu^A$, the kernel $\kappa$ being strictly positive definite.
\end{proof}

\begin{remark}Assume for a moment that both $(\mathcal A_1)$ and $(\mathcal A_2)$ hold (see Sect.~\ref{about}). As pointed out in Theorem~\ref{l-alt-countt} (see Sect.~\ref{1.7} for its proof), the characteristic property of the inner balayage $\mu^A$, given by symmetry relation (\ref{alternative}), needs then only to be verified  for certain {\it countably many\/} $\nu\in\mathcal E$, independent of the choice of $A$ and $\mu$.\end{remark}

\begin{theorem}\label{l-outer}Suppose that a locally compact space\/ $X$ has a countable base. For any\/ $\mu\in\mathcal E$ and any Borel\/ $B\subset X$, the outer balayage\/ $\mu^{*B}$ then does exist,\footnote{In contrast to that, the existence of the inner balayage $\mu^A$ has been established for {\it arbitrary\/} $X$ and $A$ (Theorem~\ref{th-intr}).} and moreover
\begin{equation*}\label{bal-eqq}\mu^{*B}=\mu^B.\end{equation*}
\end{theorem}

\begin{proof} Fix $\mu\in\mathcal E$; on account of (\ref{def-sign}), we can assume $\mu\geqslant0$.
Being locally compact and second-countable, the space $X$ is metrizable and of class $K_\sigma$, see \cite[Section~IX.2, Corollary to Proposition~16]{B3}, and hence perfectly normal, see \cite[Section~IX.1, Proposition~2]{B3}. It therefore follows by applying Theorem~\ref{l-top} that all Borel subsets of $X$ are capacitable, the kernel $\kappa$ being perfect. But, as observed in \cite[Section~9.2]{Z-arx1}, such capacitability result is exactly what we need to show that the inner balayage $\mu^B$ then serves simultaneously as the outer balayage $\mu^{*B}$, $B\subset X$ being Borel. By virtue of the uniqueness of $\mu^B$ and $\mu^{*B}$, this establishes the theorem.
\end{proof}

\section{On the topological spaces $C_0(X)$ and $\mathfrak M(X)$}\label{proof-prop}
The aim of the present section is to establish some specific properties of the topological spaces $C_0(X)$ and $\mathfrak M(X)$, to be used in the proof of Theorem~\ref{count-proof}.

Throughout this section, we do not require any of the energy, consistency, or domination principles to hold, thus treating a kernel $\kappa$ on a locally compact space $X$ as an arbitrary symmetric, l.s.c.\ function $\kappa:X\times X\to[0,\infty]$.

\subsection{On the inductive limit topology on the space $C_0(X)$}\label{sec-aux}
According to Bourbaki \cite[Section~III.1.1]{B2}, the topology on $C_0(X)$ is defined as the {\it inductive limit\/} $\mathcal T$
of the locally convex topologies of the spaces $C_0(K;X)$, where $K$ ranges over the family of all compact subsets of $X$, while $C_0(K;X)$ is the space of all $f\in C_0(X)$ with ${\rm Supp}(f)\subset K$, equipped with the topology $\mathcal T_K$ of uniform convergence on $K$. Thus, by \cite[Section~II.4, Proposition~5]{B4}, $\mathcal T$ is the {\it finest\/} of the locally convex topologies on $C_0(X)$ for which all the canonical injections $C_0(K;X)\to C_0(X)$, $K\subset X$ being compact, are continuous.
With this topology on $C_0(X)$, $\mathfrak M(X)$ becomes precisely the dual space to $C_0(X)$, see \cite[Section~III.1.3]{B2} or \cite[Section~II.4, p.~29, Example]{B4}.

\begin{lemma}[{\rm see \cite[Section~III.1, Proposition~1(i)]{B2}}]\label{aux1}For any compact\/ $K\subset X$, the topology on the space\/ $C_0(K;X)$ induced by\/ $\mathcal T$ is identical with the topology\/ $\mathcal T_K$.\end{lemma}

Assume now that $X$ is of class $K_\sigma$; this occurs, in particular, if $X$ has a countable base (see \cite[Section~IX.2, Corollary to Proposition~16]{B3}). Then there exists a sequence of relatively compact open subsets $U_j$ with the union $X$ and such that $\overline{U}_j\subset U_{j+1}$, see \cite[Section~I.9, Proposition~15]{B1}. (Here $\overline{U}_j:={\rm Cl}_XU_j$.) The space $C_0(X)$ is then the {\it strict\/} inductive limit of the sequence of spaces $C_0(\overline{U}_j;X)$, cf.\ \cite[Section~II.4.6]{B4}, for the topology induced on $C_0(\overline{U}_j;X)$ by $\mathcal T_{\overline{U}_{j+1}}$ is just $\mathcal T_{\overline{U}_j}$. Hence, by \cite[Section~II.4, Proposition~9]{B4}, the space $C_0(X)$ is Hausdorff and complete (in the topology $\mathcal T$).

\begin{lemma}\label{foot-conv}Assume that a locally compact space\/ $X$ is of class\/ $K_\sigma$, and consider a sequence\/ $(f_k)\subset C_0(X)$. Then the following two assertions are equivalent.
\begin{itemize}\item[{\rm(i)}] $(f_k)$ converges to\/ $0$ in the strict inductive limit topology\/ $\mathcal T$.
\item[{\rm(ii)}]There exists a compact subset\/ $K$ of\/ $X$ such that\/ ${\rm Supp}(f_k)\subset K$ for all\/ $k$, and\/ $(f_k)$ converges to\/ $0$ uniformly on\/ $K$.\end{itemize}
\end{lemma}

\begin{proof}Assume first that a sequence $(f_k)\subset C_0(X)$ converges to $0$ in the topology $\mathcal T$.
The set $\{f_k: k\in\mathbb N\}$ being bounded in $\mathcal T$, an application of \cite[Section~III.1, Proposition~2(ii)]{B2} shows that there is a compact set $K\subset X$ such that ${\rm Supp}(f_k)\subset K$ for all $k$. (Note that the cited Proposition can be applied here, for a locally compact space of class $K_\sigma$ is paracompact by \cite[Section~I.9, Theorem~5]{B1}.) Applying now Lemma~\ref{aux1} we therefore conclude that the sequence $(f_k)$ also converges to $0$ in the (Hausdorff) topology $\mathcal T_K$; and so indeed, (i) implies (ii). As the topology $\mathcal T$ is Hausdorff as well, the opposite implication follows directly from Lemma~\ref{aux1}.\end{proof}

\begin{remark}It is worth noting here that if a locally compact space $X$ is  {\it noncompact\/} and of class $K_\sigma$, then the strict inductive topology $\mathcal T$ is {\it strictly finer\/} than the topology of uniform convergence on $X$ (see \cite[Section~II.4, p.~29, Example]{B4}).\end{remark}

\begin{lemma}\label{foot-conv'}If a locally compact space\/ $X$ has a countable base, then there exists a countable set\/ $S\subset C_0(X)$ which is dense in\/ $C_0(X)$ equipped with the strict inductive limit topology\/ $\mathcal T$.\end{lemma}

\begin{proof}According to \cite[Section~V.3.1, Lemma]{B2}, for a locally compact sec\-ond-count\-able space $X$ there exists a countable set $S\subset C_0(X)$ having the following
property: for every function $f\in C_0(X)$, there exist a sequence $(f_j)$ of
elements of $S$ and a positive function $\varphi\in S$ such that, for every number
$\varepsilon>0$,
\[|f_j-f|<\varepsilon\varphi\text{ \ provided $j$ is large enough}.\]
Observing that for these $f\in C_0(X)$ and $(f_j)\subset S$, there exists a compact set $K\subset X$ such that all the $f$ and $f_j$ equal $0$ on $K^c$, while $f_j\to f$ uniformly on $K$, we then conclude by applying Lemma~\ref{aux1} (or Lemma~\ref{foot-conv}) that $f_j\to f$ also in $\mathcal T$.
\end{proof}

\subsection{On the vague topology on the space $\mathfrak M(X)$} We shall now discuss some specific properties of the space $\mathfrak M(X)$ for $X$ satisfying the second axiom countability.

\begin{lemma}\label{first-count}If a locally compact space\/ $X$ has a countable base, then every\/ $\xi\in\mathfrak M(X)$ has a countable base of vague neighborhoods.\end{lemma}

\begin{proof} Being metrizable (with a metric $\varrho_X$) and of class $K_\sigma$ \cite[Section~IX.2, Corollary to Proposition~16]{B3}, a locally compact se\-cond-count\-able space $X$ has a countable dense subset $(x_k)\subset X$ \cite[Section~IX.2, Proposition~12]{B3}. Therefore,
\[\Bigl\{\nu\in\mathfrak M^+: \ \int\bigl(1-k\varrho_X(x_k,x)\bigr)^+\,d|\nu-\nu_0|(x)<
1/k\Bigr\},\quad k\in\mathbb N,\]
form a countable base of vague neighborhoods of $\nu_0\in\mathfrak M^+$. (Here $|\xi|:=\xi^++\xi^-$.)\end{proof}

Thus, under assumption $(\mathcal A_1)$, the vague topology on the space $\mathfrak M$ can be described in terms of sequences, and the use of nets or filters may often be avoided. In particular, $\nu_0\in\mathfrak M^+$ belongs to the vague closure of $\mathfrak B\subset\mathfrak M^+$ if and only if there is a sequence $(\nu_k)\subset\mathfrak B$ converging vaguely to $\nu_0$. Similarly, any vaguely bounded part of $\mathfrak M^+$ contains a vaguely convergent sequence, cf. \cite[Section~III.1, Proposition~15]{B2}.

We shall now show that in the case where $(\mathcal A_2)$ holds along with $(\mathcal A_1)$, the vague convergence of a sequence $(\nu_k)\subset\mathfrak M^+$ needs only to be verified for certain {\it countably many\/} test functions $\varphi\in C_0(X)$ that are independent of the choice of $(\nu_k)$.

\subsection{Proof of Theorem~\ref{l-count}}\label{count-proof} Assume that $(\mathcal A_1)$ and $(\mathcal A_2)$ both hold. The following two observations are crucial to the proof provided below.

The first is that, in view of $(\mathcal A_1)$, a locally compact space $X$ must be of class $K_\sigma$, and hence, according to Lemma~\ref{foot-conv}, convergence of any sequence $(f_k)\subset C_0$ in the inductive limit topology $\mathcal T$ is reduced to the uniform convergence on some compact subset of $X$, depending on $X$ and $(f_k)$ only.

The second is that, according to Lemma~\ref{foot-conv'}, there exist countably many functions $g_m\in C_0$, $m\in\mathbb N$, depending on $X$ only, which form a dense subset of the space $C_0$ equipped with the inductive limit topology $\mathcal T$.

But, according to $(\mathcal A_2)$, the set $C_0\cap\{\kappa\lambda:\ \lambda\in\mathcal E\}$ is also dense in $C_0$ (in $\mathcal T$). Hence, for every $g_m$ there is a sequence $(\lambda_{g_m}^p)_{p\in\mathbb N}\subset\mathcal E$ such that $(\kappa\lambda_{g_m}^p)_{p\in\mathbb N}\subset C_0$ while
\[\kappa\lambda_{g_m}^p\to g_m\text{ \ in $\mathcal T$ as $p\to\infty$.}\]
Therefore,
\begin{equation}\mathcal E^\circ:=\bigl\{\lambda_{g_m}^p: \ m,p\in\mathbb N\bigr\}\label{e}\end{equation}
is a countable subset of $\mathcal E$, depending on $X$ and $\kappa$ only, and moreover
\[C_0^\circ:=\{\kappa\lambda:\ \lambda\in\mathcal E^\circ\}\]
is a dense subset of the space $C_0$ (equipped with the topology $\mathcal T$).

We assert that the set $\mathcal E^\circ$ thus defined is as claimed. Given a sequence $(\nu_k)\subset\mathfrak M^+$, we need to prove that $\nu_k\to\nu_0$ vaguely if (and only if)
\[\lim_{k\to\infty}\,\int\kappa\lambda\,d\nu_k=\int\kappa\lambda\,d\nu_0\text{ \ for all $\lambda\in\mathcal E^\circ$},\]
or equivalently if (and only if)
\begin{equation}\label{dens}\nu_k(\varphi)\to\nu_0(\varphi)\text{ \ for all $\varphi\in C_0^\circ$}.\end{equation}

To this end, we shall first show that for any compact $K\subset X$,
\begin{equation}\label{est}\sup_{k\in\mathbb N}\,\nu_k(K)<\infty.\end{equation} In fact, since $C_0^\circ$ is dense in $C_0$ (in $\mathcal T$), there is $\varphi_K\in C_0^\circ$ that is ${}\geqslant1_K$ on $X$; this can easily be seen with the aid of the Tie\-tze--Ury\-sohn
extension theorem \cite[Theorem~0.2.13]{E2} (applied to the normal space $X$) and Lemma~\ref{foot-conv} (cf.\ the first observation at the beginning of the present proof). Therefore, by (\ref{dens}),
\[\limsup_{k\to\infty}\,\nu_k(K)=\limsup_{k\to\infty}\,\int1_K\,d\nu_k\leqslant\lim_{k\to\infty}\,\nu_k(\varphi_K)=\nu_0(\varphi_K)<\infty,\]
and (\ref{est}) follows.

For any given $f\in C_0$, choose now a sequence $(\varphi_j)\subset C_0^\circ$ converging to $f$ in $\mathcal T$. According to Lemma~\ref{foot-conv}, there exists a compact subset $K$ of $X$ such that all the $f$ and $\varphi_j$ equal zero on $K^c$, while $\varphi_j\to f$ uniformly on $K$ as $j\to\infty$. Hence
\begin{equation}\label{sum}|\nu_k(f)-\nu_0(f)|\leqslant|\nu_k(f)-\nu_k(\varphi_j)|+|\nu_k(\varphi_j)-\nu_0(\varphi_j)|+|\nu_0(\varphi_j)-\nu_0(f)|,\end{equation}
where
\[|\nu_k(f)-\nu_k(\varphi_j)|\leqslant\int|f-\varphi_j|\,d\nu_k\leqslant\sup_{k\in\mathbb N}\,\nu_k(K)\max_{K}\,|f-\varphi_j|\]
and
\[|\nu_0(\varphi_j)-\nu_0(f)|\leqslant\nu_0(K)\max_{K}\,|f-\varphi_j|,\]
and so the first and the third summands in (\ref{sum}) approach zero as $j\to\infty$, cf.\ (\ref{est}).
Therefore, taking first $j$ sufficiently large, and then applying (\ref{dens}) to $\varphi_j$, we obtain
\[|\nu_k(f)-\nu_0(f)|\to0\text{ \ as $k\to\infty$}.\] This establishes the vague convergence of $(\nu_k)$ to $\nu_0$, whence Theorem~\ref{l-count}.

\section{Proofs of Theorems~\ref{th-m} and \ref{l-alt-countt}}\label{sec-last-1}

In the rest of the paper, the energy, consistency, and domination principles are again required to hold.

\subsection{Proof of Theorem~\ref{th-m}} Theorem~\ref{th-m} is a particular case of the following stronger assertion.

\begin{theorem}\label{cor-m} Assume that\/ $(\mathcal A_1)$ and\/ $(\mathcal A_2)$ both hold. For arbitrary\/ $A\subset X$, there exists a\/ $K_\sigma$-set\/ $A_0\subset A$ such that, for every\/ $Q$ with the property\/ $A_0\subset Q\subset A$,
\begin{equation}\label{eq-main1c}\mu^A=\mu^{Q}=\mu^{A_0}\text{ \ for all\/ $\mu\in\mathcal E$.}\end{equation}
Furthermore, for every Borel\/ $B$ such that\/ $A_0\subset B\subset A$,
\begin{equation}\label{eq-main1cor}\mu^A=\mu^{*B}=\mu^{*A_0}\text{ \ for all\/ $\mu\in\mathcal E$.}\end{equation}
\end{theorem}

\begin{proof} Fix $\mu\in\mathcal E$; on account of (\ref{def-sign}), we can certainly assume it to be positive. According to Theorem~\ref{th-intr''}(b), the net $(\kappa\mu^K)_{K\in\mathfrak C_A}$ then increases to $\kappa\mu^A$ pointwise on $X$, $K$ ranging over the upward directed set $\mathfrak C_A$ of all compact subsets of $A$. The space $X$ being Hausdorff and sec\-ond-count\-able while the functions $\kappa\mu^K$ being l.s.c.\ on $X$, we conclude by applying \cite[Appendix~VIII, Theorem~2]{Doob} that there exists an increasing sequence $(K^\mu_k)_{k\in\mathbb N}$ of compact subsets of $A$ having the property
\begin{equation*}\label{lim2}\kappa\mu^{K^\mu_k}\uparrow\kappa\mu^A\text{ \ pointwise on $X$ as $k\to\infty$}.\end{equation*}
Setting
\begin{equation}\label{prime}A_\mu':=\bigcup_{k\in\mathbb N}\,K^\mu_k,\end{equation}
we see, again by Theorem~\ref{th-intr''}(b), that
\begin{equation}\label{prime22}\kappa\mu^{A_\mu'}=\lim_{k\to\infty}\,\kappa\mu^{K^\mu_k}=\kappa\mu^A\text{ \ on\ }X,\end{equation}
and hence, by virtue of (\ref{en}),
\[\mu^{A_\mu'}=\mu^A.\]

Even more generally, for any set $L$ such that $A_\mu'\subset L\subset A$,
\begin{equation}\label{prime2}\mu^{A_\mu'}=\mu^L=\mu^A,\end{equation}
for
\[\kappa\mu^{A_\mu'}\leqslant\kappa\mu^L\leqslant\kappa\mu^A=\kappa\mu^{A_\mu'}\text{ \ on\ }X,\]
the inequalities being true by monotonicity property (\ref{mon}) while the equality by (\ref{prime22}).

We assert that the set $A_0$ we are looking for can be defined by means of the formula
\begin{equation}\label{def-A}A_0:=\bigcup_{\lambda\in\mathcal E^\circ}\,\bigl(A'_{\lambda^+}\cup A'_{\lambda^-}\bigr),\end{equation}
where $\mathcal E^\circ$ is the (countable) family of (signed) measures $\lambda\in\mathcal E$ appearing in Theorem~\ref{l-count} (see (\ref{e})), whereas $A'_{\lambda^\pm}$ is given by (\ref{prime}) with $\mu:=\lambda^\pm$.

In fact, $A_0$ is a $K_\sigma$-sub\-set of $A$, for so is every $A'_{\lambda^\pm}$, and moreover
\[(\lambda^\pm)^{A_0}=(\lambda^\pm)^A\text{ \ for all\ }\lambda\in\mathcal E^\circ,\]
by (\ref{prime2}) with $\mu:=\lambda^\pm$ and $L:=A_0$. In view of (\ref{def-sign}), this gives
\begin{equation*}\label{L}\lambda^{A_0}=\lambda^A\text{ \ for all\ }\lambda\in\mathcal E^\circ.\end{equation*}
Applying Theorem~\ref{l-alt} with $\nu:=\lambda\in\mathcal E^\circ$, we therefore obtain
\[\kappa(\mu^A,\lambda)=\kappa(\mu,\lambda^A)=\kappa(\mu,\lambda^{A_0})=\kappa(\mu^{A_0},\lambda)\text{ \ for all\ }\lambda\in\mathcal E^\circ,\]
and hence
\[\int\kappa\lambda\,d\mu^A=\int\kappa\lambda\,d\mu^{A_0}\text{ \ for all $\lambda\in\mathcal E^\circ$},\]
which according to the latter part of Theorem~\ref{l-count} implies the equality
\begin{equation}\label{3.1}\mu^A=\mu^{A_0}.\end{equation}

Similarly as above, we now conclude from (\ref{3.1}) by use of (\ref{mon}) that for any $Q$ such that $A_0\subset Q\subset A$,
\[\kappa\mu^{A_0}\leqslant\kappa\mu^Q\leqslant\kappa\mu^A=\kappa\mu^{A_0}\text{ \ on\ }X,\]
hence $\mu^{A_0}=\mu^{Q}=\mu^A$, again by (\ref{en}), and consequently (\ref{eq-main1c}), by noting that the set $A_0$ is independent of the choice of $\mu$ (see (\ref{def-A})).

According to Theorem~\ref{l-outer}, $\mu^B=\mu^{*B}$ for every Borel $B\subset X$, the space $X$ being second-countable. Combining this with (\ref{eq-main1c}) yields (\ref{eq-main1cor}).\end{proof}

\subsection{Proof of Theorem~\ref{l-alt-countt}}\label{1.7} Let $(\mathcal A_1)$ and $(\mathcal A_2)$ be both satisfied, and let $\mathcal E^\circ\subset\mathcal E$ be the set appearing in Theorem~\ref{l-count} (see (\ref{e})). Given $\mu\in\mathcal E$ and $A\subset X$, (\ref{rel1}) obviously holds by (\ref{alternative}) with $\nu:=\lambda\in\mathcal E^\circ$. To show that (\ref{rel1}) determines the inner balayage $\mu^A$ uniquely within $\mathcal E$, suppose now that (\ref{rel2}) takes place for some $\xi\in\mathcal E$. Subtracting (\ref{rel2}) from (\ref{rel1}) results in the relation
\[\int\kappa\lambda\,d\xi=\int\kappa\lambda\,d\mu^A\text{ \ for all $\lambda\in\mathcal E^\circ$},\]
which by the latter part of Theorem~\ref{l-count} gives $\xi=\mu^A$.
Finally, the remaining assertion of the theorem is obtained from the above by applying Theorem~\ref{l-outer}.

\section{Convergence of swept measures (application of Theorem~\ref{cor-m})}

\begin{theorem}\label{pr-cont}Assume that\/ $(\mathcal A_1)$ and\/ $(\mathcal A_2)$ are both satisfied,\footnote{As follows from the proof (see Step~1), in the case where $A\subset X$ is Borel, Theorem~\ref{pr-cont}(a) remains valid even if neither of $(\mathcal A_1)$ and $(\mathcal A_2)$ holds (cf.\ \cite[Theorem~4.9]{Z-arx1}).} and consider the exhaustion of arbitrary\/ $A\subset X$ by
\begin{equation*}\label{exh}A_j:=A\cap U_j,\quad j\in\mathbb N,\end{equation*}
$(U_j)$ being an increasing sequence of Borel
sets with the union\/ $X$. {\rm(}Such\/ $(U_j)$ exists, the topology on\/ $X$ having a countable base.{\rm)} Then the following two assertions hold.
\begin{itemize}\item[{\rm(a)}] For any\/ $\mu\in\mathcal E$,
\begin{align}\label{cont1}&\mu^{A_j}\to\mu^A\text{ \ strongly and vaguely},\\
\label{cont2}&\kappa\mu^{A_j}\to\kappa\mu^A\text{ \ pointwise q.e.\ on\ }X.
\end{align}
\item[{\rm(b)}]There are Borel sets\/ $B_j$ such that\/ $B_j\subset A_j$ and\/ $B_j\subset B_{j+1}$ for each\/ $j\in\mathbb N$, and for all\/ $\mu\in\mathcal E$,
\begin{align*}&\mu^{*B_j}\to\mu^A\text{ \ strongly and vaguely},\\
&\kappa\mu^{*B_j}\to\kappa\mu^A\text{ \ pointwise q.e.\ on\ }X.
\end{align*}
\end{itemize}
\end{theorem}

\begin{proof} Recall that for any $\nu\in\mathcal E$, $\kappa\nu$ is well defined and finite q.e.\ on $X$ (Sect.~\ref{sec}). On account of definition (\ref{def-sign}) and the countable subadditivity of outer capacity \cite[Lemma~2.3.5]{F1}, we can therefore assume throughout the proof that $\mu$ is positive.

The proof will be given in three steps.\smallskip

{\bf Step~1.} Suppose first that $A$ is Borel; then so are all the sets $A_j$. Given $\mu\in\mathcal E^+$, the inner balayage $\mu^{A_j}$ is, in fact, the orthogonal projection of $\mu$ onto $\mathcal E_{A_j}'$, the strong closure of $\mathcal E^+_{A_j}$, cf.\ Theorem~\ref{th-intr}(b). Thus $\mu^{A_j}\in\mathcal E_{A_j}'$ and
\[\|\mu-\mu^{A_j}\|=\min_{\nu\in\mathcal E_{A_j}'}\,\|\mu-\nu\|=\varrho(\mu,\mathcal E_{A_j}'),\]
where
\[\varrho(\mu,\mathcal B):=\inf_{\nu\in\mathcal B}\,\|\mu-\nu\|\text{ \ for\ }\mathcal B\subset\mathcal E^+.\]
Since obviously $\mathcal E_{A_j}'\subset\mathcal E_{A_p}'\subset\mathcal E_A'$ for any $j\leqslant p$, applying \cite[Lemma~4.1.1]{F1} with $\mathcal H:=\mathcal E$, $\Gamma:=\{\mu-\nu:\ \nu\in\mathcal E_{A_p}'\}$, and $\lambda:=\mu-\mu^{A_p}$ yields
\begin{equation}\label{Last}\|\mu^{A_j}-\mu^{A_p}\|^2=\|(\mu-\mu^{A_j})-(\mu-\mu^{A_p})\|^2\leqslant\|\mu-\mu^{A_j}\|^2-\|\mu-\mu^{A_p}\|^2.\end{equation}
Being decreasing and lower bounded, the sequence $(\|\mu-\mu^{A_j}\|)$ is Cauchy in $\mathbb R$, which together with (\ref{Last}) shows that the sequence $(\mu^{A_j})$ is strong Cauchy in $\mathcal E^+$, and hence it converges strongly and vaguely to the unique $\mu_0\in\mathcal E_A'$, $\mathcal E_A'$ being strongly closed while $\mathcal E^+$ strongly complete. This implies that
\begin{align}\notag\varrho(\mu,\mathcal E_A^+)&=\varrho(\mu,\mathcal E_A')\leqslant\|\mu-\mu_0\|=\lim_{j\to\infty}\,\|\mu-\mu^{A_j}\|\\{}&=\lim_{j\to\infty}\,\varrho(\mu,\mathcal E_{A_j}')=
\lim_{j\to\infty}\,\varrho(\mu,\mathcal E^+_{A_j}),\label{stream}\end{align}
the first and last equalities being evident by definition.

The sets $A_j$ being universally measurable, for every $\nu\in\mathcal E_A^+$ we get
\[\lim_{j\to\infty}\,\nu|_{A_j}(f)=\lim_{j\to\infty}\,\int1_{A_j}f\,d\nu=\int1_Af\,d\nu=\nu_A(f)=\nu(f)\text{ \ for all\ }f\in C_0^+,\]
where the second equality holds by \cite[Section~IV.1, Theorem~3]{B2}. (Here, as usual, $C_0^+$ stands for the cone of all positive $f\in C_0$.)
Hence $\nu|_{A_j}\to\nu$ vaguely as $j\to\infty$, and therefore, by the principle of descent,
\begin{equation}\label{dis}\|\nu\|\leqslant\liminf_{j\to\infty}\,\|\nu|_{A_j}\|,\quad\kappa(\mu,\nu)\leqslant\liminf_{j\to\infty}\,\kappa(\mu,\nu|_{A_j}).\end{equation}
On the other hand, since $\nu|_{A_{j+1}}-\nu|_{A_j}\geqslant0$ for every $j$, 
\[\|\nu\|\geqslant\limsup_{j\to\infty}\,\|\nu|_{A_j}\|,\quad\kappa(\mu,\nu)\geqslant\limsup_{j\to\infty}\,\kappa(\mu,\nu|_{A_j}),\]
which together with (\ref{dis}) implies that
\[\|\mu-\nu\|=\lim_{j\to\infty}\,\|\mu-\nu|_{A_j}\|\geqslant\lim_{j\to\infty}\,\varrho(\mu,\mathcal E^+_{A_j})\text{ \ for every\ }\nu\in\mathcal E^+_A,\]
and consequently
\[\varrho(\mu,\mathcal E_A^+)\geqslant\lim_{j\to\infty}\,\varrho(\mu,\mathcal E^+_{A_j}).\]
Combining this with (\ref{stream}) proves that $\mu_0$, the strong and vague limit of the sequence $(\mu^{A_j})$, is equal to the orthogonal projection of $\mu$ onto $\mathcal E_A'$, and hence, in fact, to the inner balayage $\mu^A$, cf.\ Theorem~\ref{th-intr}(b). This establishes (\ref{cont1}) for Borel $A$.\smallskip

{\bf Step~2.} Let $A$ now be arbitrary. As an application of Theorem~\ref{cor-m}, observe the existence of $K_\sigma$-sets $A'\subset A$ and $A_j'\subset A_j$, $j\in\mathbb N$, such that $A_j'\subset A_{j+1}'$ and moreover,
\[\mu^{A'}=\mu^A,\quad\mu^{A_j'}=\mu^{A_j}.\]
Writing $\check{A}_j:=A'\cap U_j$, we have $\check{A}_j\subset A_j$ and hence, again by Theorem~\ref{cor-m},
\begin{equation}\label{contt22}\mu^{A_j'\cup\check{A}_j}=\mu^{A_j'}=\mu^{A_j}\text{ \ for all\ }j\in\mathbb N.\end{equation}
Similarly,
\begin{equation}\label{contt11}\mu^{\tilde{A}}=\mu^{A'}=\mu^A,\end{equation}
where
\[\tilde{A}:=\bigcup_{j\in\mathbb N}\,(A_j'\cup\check{A}_j).\]
The sets $A_j'\cup\check{A}_j$, $j\in\mathbb N$, being Borel and forming an increasing sequence with the union $\tilde{A}$, we conclude from what has been proved above (see Step~1) that
\[\mu^{A_j'\cup\check{A}_j}\to\mu^{\tilde{A}}\text{ \ strongly and vaguely},\]
which in view of (\ref{contt22}) and (\ref{contt11}) establishes (\ref{cont1}) for $A$ arbitrary.

By monotonicity property (\ref{mon}), the potentials $\kappa\mu^{A_j}$ increase pointwise on $X$ (as $j$ ranges over $\mathbb N$) and do not exceed $\kappa\mu^A$, whence
\[\lim_{j\to\infty}\,\kappa\mu^{A_j}\leqslant\kappa\mu^A\text{ \ on $X$}.\]
The opposite being concluded from (\ref{cont1}) by the principle of descent, (\ref{cont2}) follows.
This completes the proof of assertion (a).\smallskip

{\bf Step~3.} Observe that all the sets $A_j'$ and $\check{A}_j$ (see Step~2) are independent of the choice of $\mu\in\mathcal E^+$ (cf.\ Theorem~\ref{cor-m}).
Therefore, applying Theorem~\ref{l-outer} to the (Borel) sets $B_j:=A_j'\cup\check{A}_j$, $j\in\mathbb N$, we derive (b) from relations (\ref{cont1}), (\ref{cont2}), and (\ref{contt22}).
\end{proof}

\section{An open question}\label{sec-open}

In view of the results thus obtained, the following open question naturally arises (compare with the theory of inner Newtonian and Riesz balayage, developed in \cite{Ca2} and \cite{Z-bal,Zor21}, respectively):

\begin{problem} What kind of additional assumptions on a locally compact space\/ $X$ and a kernel\/ $\kappa$ would make it possible to generalize the theory of inner balayage, developed in\/ {\rm\cite{Z-arx1}} and the present paper, to Radon measures on\/ $X$ of infinite energy?\end{problem}

\section{Acknowledgements} The author thanks Krzysztof Bogdan for helpful discussions around the paper~\cite{Bogdan}, and Bent Fuglede for valuable comments on the manuscript.

\end{document}